\documentclass[lettersize, journal]{IEEEtran}

\usepackage{amsmath,amssymb} 
\usepackage{amsthm}
\usepackage{adjustbox}
\newcommand{\R}{\mathbb{R}}
\usepackage{graphicx}
\usepackage{caption}
\usepackage{subcaption}
\usepackage[linesnumbered,ruled]{algorithm2e}
\usepackage[justification=centering]{caption}
\usepackage{flushend}

\newtheorem{lemma}{Lemma}
\usepackage{flushend}

\DeclareMathOperator*{\argmin}{argmin}

\allowdisplaybreaks

\title{\LARGE \bf
Noise Aware Path Planning and\\ Power Management of Hybrid Fuel UAVs
}

\author{Drew Scott$^{1}$, Satyanarayana G. Manyam$^{2}$, Isaac E. Weintraub$^{3}$, David W. Casbeer$^{4}$, \\Manish Kumar$^{5}$   
\thanks{$^{1}$Department of Mechanical and Materials Engineering, University of Cincinnati 
        {\tt\small scott2dd@mail.uc.edu}}%
\thanks{$^{2}$Research Scientist, Infoscitex corp., a DCS company, Dayton OH, 45431
        {\tt\small msngupta@gmail.com}}%
\thanks{$^{3}$Electronics Engineer, Control Science Center, Air Force Research Laboratory, WPAFB, OH, 45433 {\tt\small isaac.weintraub.1@us.af.mil}}%
\thanks{$^{4}$Technical Area Lead, Cooperative \& Intelligent Control, Control Science
Center, Air Force Research Laboratory, WPAFB, OH, 45433
        {\tt\small david.casbeer@us.af.mil}}%
\thanks{$^{5}$Professor in Department of Mechanical and Materials Engineering, University of Cincinnati {\tt\small kumarmu@ucmail.uc.edu}}%
\thanks{This material is in part based on research sponsored by the Ohio Department of Higher Education and the Southwestern Council for Higher Education under Ohio House Bill 49 of the 132nd General Assembly. The U.S. Government is authorized to reproduce and distribute reprints for Governmental purposes notwithstanding any copyright notation thereon. The views and conclusions contained herein are those of the authors and should not be interpreted as necessarily representing the official policies or endorsements, either expressed or implied, of Southwestern Council for Higher Education, the Ohio Department of Higher Education or the U.S. Government.}
\thanks{APPROVED for public release: distribution unlimited, case number: AFRL-2023-5566.} 
}

\begin{document}
\maketitle
\thispagestyle{empty}
\pagestyle{empty}
\begin{abstract}
Hybrid fuel Unmanned Aerial Vehicles (UAV), through their combination of multiple energy sources, offer several advantages over the standard single fuel source configuration, the primary one being increased range and efficiency. Multiple power or fuel sources also allow the distinct pitfalls of each source to be mitigated while exploiting the advantages within the mission or path planning. We consider here a UAV equipped with a combustion engine-generator and battery pack as energy sources. We consider the path planning and power-management of this platform in a noise-aware manner. To solve the path planning problem, we first present the Mixed Integer Linear Program (MILP) formulation of the problem. We then present and analyze a label-correcting algorithm, for which a pseudo-polynomial running time is proven. Results of extensive numerical testing are presented which analyze the performance and scalability of the labeling algorithm for various graph structures, problem parameters, and search heuristics.  It is shown that the algorithm can solve instances on graphs as large as twenty thousand nodes in only a few seconds. 
\end{abstract}


\textit{Note to Practitioners:} \textbf{
The problem and algorithms proposed in this paper are relevant to constrained planning problems in general and specifically to those focused on widespread usage of small aerial vehicles in congested, urban environments. We are concerned here with the path planning of hybrid-fuel aerial vehicles in a noise-aware manner. This is motivated by the increasing usage of aerial vehicles, envisioning a probable future restriction on noise production in certain airspaces and the planning of such vehicles in those airspaces.  We explore this novel problem, and present an approach to quickly find the optimal path and power plan in the presence of the noise constraints.  The approach here is a discrete one, where the solution is a discrete set of edges and discrete generator settings which must be smoothed to obtain control inputs for a real system.  The discrete approach allows solutions to be found quickly while giving up the true optimal trajectory that can be found when considering from a continuous framework.  In practice, environment sampling and graph construction will greatly affect time-to-solve and as overall solution quality relative to a continuous approach to the trajectory and generator control.}


\section{Introduction}\label{sec:intro}
Hybrid fuel Unmanned Aerial Vehicles (UAVs) are those in which multiple fuel sources are used for energy storage and/or power~\cite{TOWNSEND2020e05285}. Generally, the primary motivation for hybridization is increased endurance and efficiency.  However, by intelligently switching the power modality throughout the mission, other less obvious advantages can be exploited. For example, a hybrid-fuel platform could shut down the noisy gas powered engine and utilize the electric motor, allowing the UAV to operate in low-altitude areas where noise is undesirable~\cite{yedavalli2019assessment,torija2020effects,cussen2022uav}.

We consider series hybrid UAVs~\cite{TOWNSEND2020e05285} motivated by two scenarios: i) Noise-sensitive surveillance missions, and ii) Widespread use of UAVs in residential and business environments. In both cases, noise and flight range are important. In case of the latter scenario, with this increase use of UAVs will be an increase in ground-level noise in both frequency and intensity.  Thus we envision in the near future location-dependent noise restrictions on low-altitude aircraft.  A series hybrid UAV uses an electric motor powered by a battery pack, which is recharged through a a genset (gas engine and generator). The combustion engine, with its high power and energy density, provides extended flight range when compared to battery storage alone but at the expense of increased noise. In contrast, the battery allows quieter operation but has lower energy density. This paper addresses a coupled path and power planning problem which arises when noise-restrictions are placed on the aforementioned hybrid platform.


As shown in \cite{moshkov2021study}, the main noise source of a UAV with a shrouded propeller is the two-stroke engine powered onboard for power. Thus, a hybrid UAV with gas and battery as fuel sources has both long endurance and the ability to operate in noise-restricted areas by running in battery-only mode. In these scenarios, the path planning problem is coupled with the power management problem. The power management problem must define generator switching to ensure the battery is never depleted and that the generator is not run in noise-restricted zones. The aim of the path planning problem is to find a path of minimum cost while avoiding obstacles.  However, determination of power plan depends on the path due to the noise-restricted zones, and determination of path depends on a feasible power plan existing for the path of question. Thus, the path planning is coupled to the power planning, where the path and power plan must be found in tandem. Due to this complexity, traditional path and power planning methods cannot be applied directly and must be approached as a new problem. In this paper we are concerned with the coupled planning problem, which we refer to as the Noise-Restricted Hybrid Fuel Shortest Path Problem (NRHFSPP).

At some point in the overall planning process, a trajectory subject to vehicle dynamics through the free-airspace must be found. Often, a geometric path is first found that is later processed into a flyable trajectory \cite{gasparetto2012trajectory}. While this smoothed trajectory in general will not be the same as that obtained by solving the continuous trajectory optimization problem directly, the assumption made is this will be on average a tight bound on the true optimal trajectory, depending on the discretization scheme used to construct the graph. However, this two-stage approach is generally much faster to solve than finding the optimal trajectory directly. Of concern in this paper is finding a geometric path rather than a flyable trajectory. Thus, we pose this problem as a discrete optimization problem, where the airspace is discretized into a graph and each edge in the graph is parameterized by cost and energy values needed to define the NRHFSPP, described later.

This work focuses on planning algorithms for the NRHFSPP when posed on a graph and analyzes the performance and scalability of these algorithms. The planning algorithms presented in this paper will work on any graph, given positive edge costs, and thus can be used in combination with many graph construction techniques.  We are not focused on efficient graph construction techniques for a given airspace, examples of such techniques including visibility graphs \cite{scholer2011configuration}, Delaunay triangulation \cite{fortune1995voronoi}, Voronoi diagrams \cite{aurenhammer1991voronoi}, or Steiner point additions \cite{erten2009quality}.  Rather, the focus is fast planning algorithms to work on any general graph.

The paper is ordered as follows. A review of related prior work in the literature is given in Section \ref{sec:lit}. Section \ref{sec:prob_form} gives the specific formulation of the NRHFSPP and Section \ref{sec:MILP} gives the MILP definition of the problem, proof the problem is NP-hard, and proves a fast lower-bound. Section \ref{sec:labeling} presents a labeling algorithm developed to solve the NRHFSPP, proof of its exactness, and derives its time-complexity. The results of extensive numerical testing are given in Section \ref{sec:results}, evaluating two label selection methods for a variety of graph types and sizes. Finally, Section \ref{sec:conc} gives conclusions and directions for future work. 

\section{Literature Review}\label{sec:lit}

The majority of literature addressing planning for alternative fuel UAVs focuses on solar powered~\cite{klesh2007energy, klesh2009solar, hosseini2016energy}, fuel-cell~\cite{dudek2013hybrid, mobariz2015long}, or a combination of the two~\cite{dobrokhodov2020energy}. Generally, these papers seek for methods to improve flight endurance through energy management or path planning. As such they will not directly address the NRHFSPP.

Fuel constraints and refueling events have been considered before in case of the Vehicle Routing Problem (VRP)  variants ~\cite{sundar2016formulations, sundar2017analysis}. The VRP has been studied specifically with Electric Vehicles (EV), referred to as the Electric VRP (EVRP).  An extensive survey of the problems and solution approaches is given in \cite{erdelic2019survey}. Other studies on hybrid vehicle routing are found in  \cite{verma2018electric, doppstadt2016hybrid, vincent2017simulated, hiermann2019routing}. UAV routing with refueling locations is studied in \cite{sundar2013algorithms} and with battery recharging in \cite{alyassi2022autonomous}. A scheduling problem pertaining to UAV fueling events is discussed in \cite{jin2006optimal}, and a similar problem with respect to battery recharge scheduling is given in \cite{mathew2015multirobot}. Such approaches will not solve the NRHFSPP as none deal directly with a path planning problem subject to fuel/battery constraints while solving for a power management plan which is coupled to the path planning.

Constrained path planning is another area relevant to the NRHFSPP. The standard Shortest Path Problem (SPP) is the simplest version and is well known to be solved in polynomial-time. Variants of the SPP generalize the problem to include additional constraints, such the  Constrained SPP (CSPP), Resource Constrained SPP (RCSPP) and RCSPP with Replenishment (RCSPP-R). These are in general NP-Hard.  In general, the time-to-solve these SPP variants depends heavily on the parameters which define the resource constraints. Some manner of  restriction on the problem parameters allows pseduo-polynomial time algorithms to emerge. With these algorithms, the problems can often be scaled up to very large graphs and solved in a reasonable time. 

The NRHFSPP is reminiscent of the resource constrained shortest path problem (RCSPP) when considering the battery and fuel as resources which must never be fully depleted. The standard RCSPP is to find the shortest path between a start and goal node while satisfying constraints on resources which are consumed along the path. The resource constraints can take the form of time-windows, such that  when reaching a given node, the resource values must be within a certain range. The constraints may also be in a capacity constraint form, where the resources are constrained to a maximum across the whole path as they rise throughout travel.  A survey of exact approaches to the RCSPP is given in \cite{pugliese2013survey}. 

Recent work on these problems mainly has been focused on acceleration of existing frameworks. Main approaches to solve these types of problems include: i) Graph processing techniques, where edges and nodes may be removed and bounds placed on nodes with respect to partial paths ii) Lagrange relaxation iii) Dynamic programming, often in the form of label correcting or label setting algorithms iv) recursive algorithms, which propagate partial paths through the graphs to reach the goal node. A series of studies on a recursive ``pulse algorithm" by Lozano et. al. showed great success for RCSPP and similar problems \cite{cabrera2020exact, lozano2016exact, lozano2013exact}. 

In the standard RCSPP, resources cannot be replenished and are always depleting along the path. The RCSPP with replenishment (RCSPP-R) is a variant in which resources consumed along certain arcs and replenished along others. Such a constraint is relevant in crew or aircraft scheduling where a shift switch or maintenance event can replenish a reserve. In network design, signals may need to be amplified/replenished, where visiting certain nodes or arcs functions as signal replenishment. Two recent studies on the RCSPP-R are given in \cite{bolivar2014acceleration} and \cite{smith2012solving}. As pointed out in \cite{bolivar2014acceleration}, the RCSPP-R has received little attention relative to the RCSPP despite the relevance of the replenishment element to real-world problems. However, in the RCSPP-R there is no explicit decision variable for replenishment but is implicit based on the edges used in the path. The NRHFSPP differs from the RCSPP-R in that, while replenishment is allowed, it is dependent on a secondary decision variable rather than being implicit based on which nodes/edges are visited. This increases problem complexity as there is not only a path to be found but also the operation mode of the UAV along that path. 

A prior work done by Cabral et al. aligns closely with the NRHFSPP, referred to by Cabral as the Network Design Problem with Relays (NDPR). Here, the problem is to find a network of minimal cost which connects nodes in a community to an origin node. Relays can be added at nodes to replenish signal strength completely, with the relay placement being a decision variable. There is a maximal distance a signal can travel before requiring a relay. Relay placement incurs a cost in the objective, thus constraining relay placement to the end of both feasibility and optimality. Their work on the NDPR is presented in \cite{cabral2007network, cabral2008wide, cabralPhD}. This work aligns closely with the NRHFSPP primarily due to the secondary decision variable for relay placement. However, our exact applications and problem formulation sufficiently differs and requires a novel approach tailored to the NRHFSPP.

There are a few prior references that look at the NRHFSPP or a very similar problem. One study \cite{manyam2022path} sought to address the NRHFSPP directly by developing a guaranteed a lower-bound on the solution. In \cite{jadischke2023optimal}, a similar problem is presented which searches for the energy optimal path. Lastly, a preliminary version of NRHFSPP and approaches to solve it were published in \cite{scott2022hybrid} which is generalized here. The contributions of this paper beyond the prior work include: i) Proof that the NRHFSPP is NP-Hard, ii) Acceleration of the labeling algorithm to solve the NRHFSPP iii) Proof of the SPP as a lower bound iv) Proof of exactness of the labeling algorithm presented v) derivation of pseudo-polynomial running time for the labeling algorithm, vi) Additional selection methods, node selection and label selection, in the labeling algorithm, vii) Extension of the NRHFSPP to include battery drain due to engine startup and gliding edges defined in graph construction, viii) More extensive numerical evaluation of the presented algorithms.

\section{Problem Formulation}\label{sec:prob_form}
\begin{figure}
     \centering
     \begin{subfigure}[b]{0.2\textwidth}
         \centering
         \includegraphics[width=\textwidth]{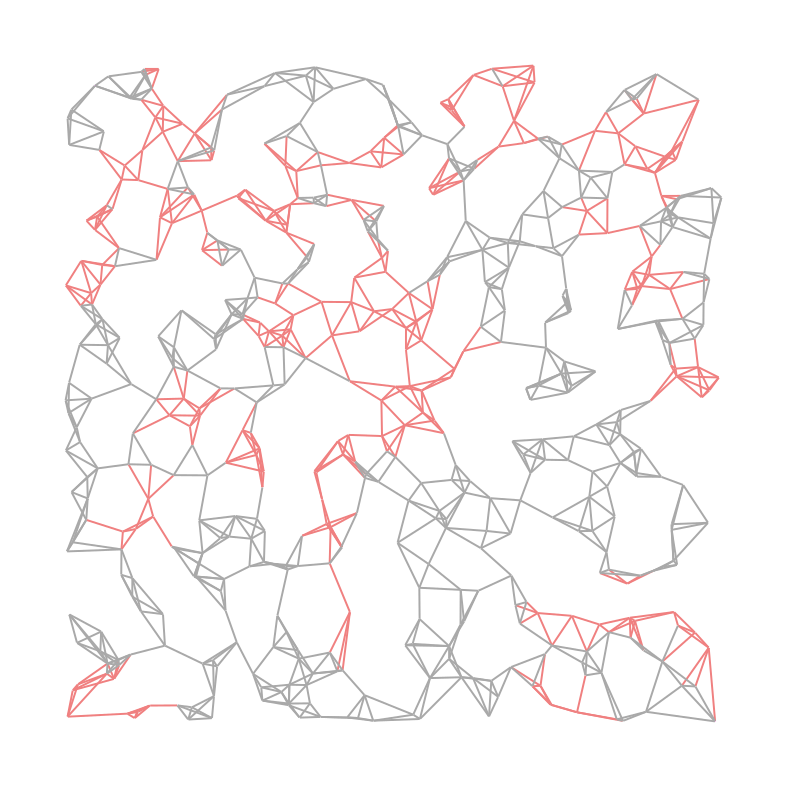}
         \caption{Euclidean Graph}
         \label{fig:graph_euc}
     \end{subfigure}
     \begin{subfigure}[b]{0.2\textwidth}
         \centering
         \includegraphics[width=\textwidth]{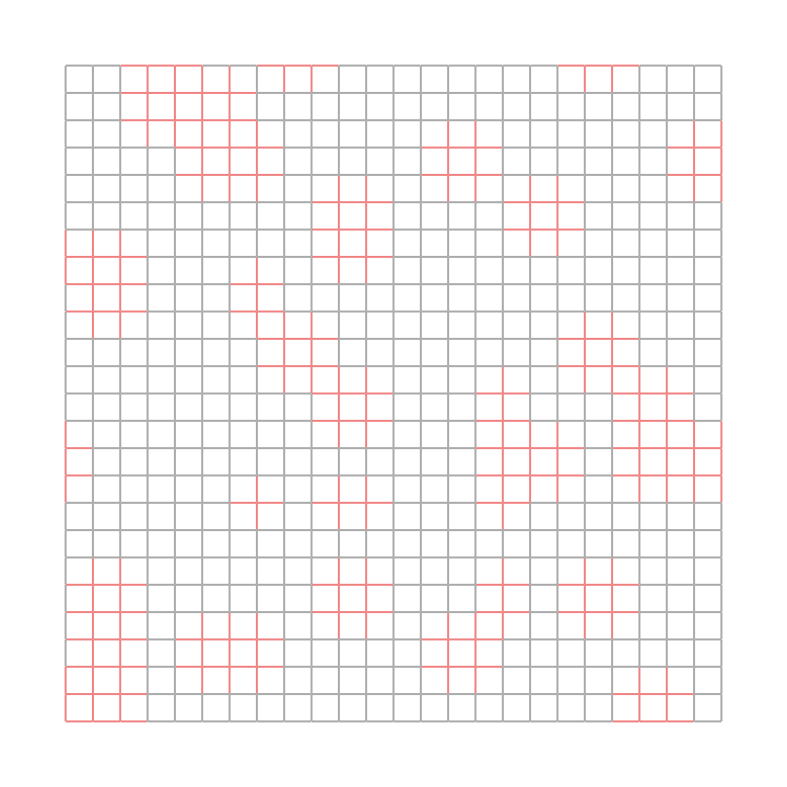}
         \caption{Lattice Graph}
         \label{fig:graph_lattice}
     \end{subfigure}
     \hfill
     \caption{Graph Examples. Noise Restricted Edges in Red.}
     \label{fig:three graphs}
\end{figure}

In this section we thoroughly define the NRHFSPP; the MILP formulation of the same is presented in Section \ref{sec:MILP}. An instance of the problem constitutes the initial UAV battery state $B_0 \in R$, fuel level $Q_0 \in \R$, initial node ($S$), goal node ($T$), and a parameterized graph represented as a set of nodes and edges $(N,E)$. While a node represents a geographical location, an edge represents the path or transition between two nodes. The graph parameters defined for each edge $(i,j) \in E$ are: (i) Edge noise restriction $G_{ij} \in \{0,1\}$, (ii) Edge cost $D_{ij} \in \R$, (iii) Battery consumption $C_{ij}\in \R$ (iv) Recharge of the battery from generator $Z_{ij} \in \R$. A startup cost $V \in \R$, constant across the edges, is also defined to account for the battery drain required to start up the generator. The NRHFSPP aims to find the minimum cost path from $S$ to $T$ and a generator schedule while satisfying noise restrictions, battery state of charge constraints, and fuel level constraints. The battery and fuel levels are constrained to be above a minimum value.  Furthermore, the battery is also restricted to be below a maximum charge level. 

AS described above, in surveillance or urban path-planning applications, there may exist sections of airspace where noise levels are restricted. These noise restrictions translate to constraints on the generator such that the flight mode is restricted to electric-only along those edges. The generator must be in an ``off-state'' while traveling along such edges with noise restrictions.

A maximum fuel state is not enforced as the fuel level never increases along the path -- the fuel level is monotonically non-increasing in this formulation. However, should this problem be expanded to include refueling events, a max fuel state would need to be constrained; this is left as future work as vehicle refueling is not considered herein. The battery charge is expended in two ways: i) discharge while traveling due to  motors power draw ii) discharge at start up of the generator. The start-up battery drain includes both the cost of starting the generator engine as well as any energy needed to heat the engine.\footnote{Starting an engine at altitude requires heating because the temperature at altitude is too cold to support ignition.} 

Along certain edges, where the altitude decreases, we allow the UAV to glide such that there is no power draw from the battery to the motors, under the assumption that the UAV peripherals draw negligible power. The generator can still be run along these gliding edges to recharge the battery. These gliding maneuvers can be performed on the edges with or without noise-restrictions. 

With regard to the gliding edges, it is always advantageous for the UAV to glide whenever allowed because this will use minimal battery energy for the same travel cost compared to flying with motor on. As a result, this affects the energy cost on these edges since there is no power needed for the motors. This allows a third energy storage modality: the potential energy stored by gaining altitude. This altitude-gain can later be utilized in noise restricted air-space  by gliding down and with significantly lower energy cost, thus enabling longer flight duration through the noise-restricted airspace. This potential energy is not tracked, but rather is implicit based on the path.

We consider two types of graphs for the computational experiments; examples of these two types of graphs are illustrated in Figures \ref{fig:graph_euc} and \ref{fig:graph_lattice}, where noise restricted edges are shown in red. The former is referred to as a Euclidean graph and the latter a lattice graph. The construction of these graphs is discussed in Section \ref{sec:results}. Computational results differ slightly between these two types of graph structures and are presented in Section \ref{sec:results}. 

\begin{figure}
\centering
\begin{subfigure}{0.35\textwidth}
    \centering
    \includegraphics[width = .8\textwidth]{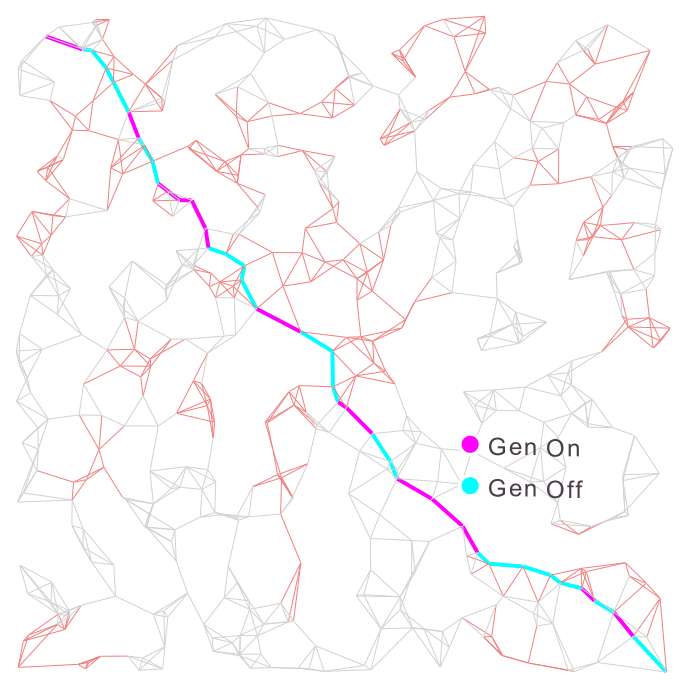}
    \caption{Path and Generator Solution}
    \label{fig:example}
\end{subfigure}
\begin{subfigure}{0.38\textwidth}
    \includegraphics[width = \textwidth]{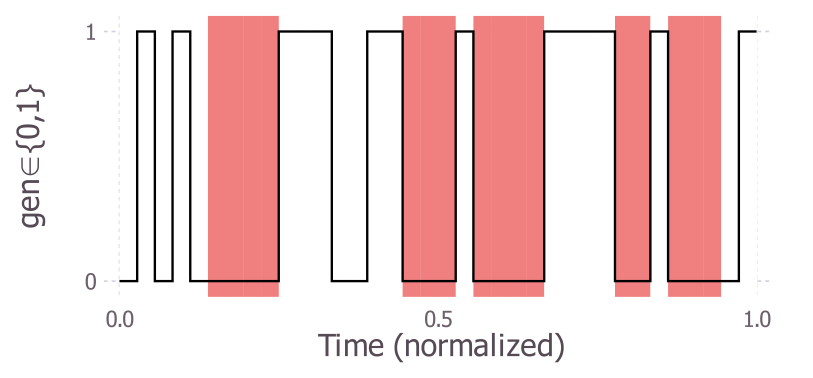}
    \caption{Generator Solution}
    \label{fig:example_gen_NR}
\end{subfigure}
\caption{Example Solution - 2D Euclidean Graph - 550 Nodes}
\end{figure}

A solution to the NRHFSPP includes both the minimum cost path and associated generator schedule. 
 An example solution is shown in Figures \ref{fig:example} and \ref{fig:example_gen_NR}. The generator schedule, shown in Fig. \ref{fig:example_gen_NR}, is overlaid with the noise restrictions (red bands) that occur along the path.

\section{MILP Formulation}\label{sec:MILP}
\subsection{Decision Variables}

$x_{ij} \in \{0,1\}$ - is edge $(i,j)$ used in the solution

$g_{ij} \in \{0,1\}$ - is generator run on the edge $(i,j)$


$b_j \in \R{}$ - battery charge at node $j$

$q_j \in \R{}$ - fuel level at node $j$

\subsection{Constant Parameters}

$D_{ij} \in \R$ - cost of edge $(i,j)$

$C_{ij} \in \R$ - energy cost (battery discharge) of edge $(i,j)$

$Z_{ij} \in \R$ - battery recharge by generator along edge $(i,j)$

$G_{ij} \in \{0,1\}$ - is generator allowed to run along edge $(i,j)$ 


$V \in \R$ - Battery charge required to start generator

$N$ - set of nodes

$E$ - set of edges 

$S$ - start node

$T$ - goal node

$M$ - large constant value

$Q_0$ - initial fuel level

$B_{min}$ - minimum allowed battery charge

$B_0$ - initial battery state

$B_{max}$ - maximum battery state

$\delta(i)$ - set of all nodes connected to node $i$

\subsection{Objective and Constraints}
\begin{align}
     & J^* = \min  \textstyle\sum_{(i,j) \in E} D_{i,j} x_{i,j} \label{MILPobj}\\
    & \textstyle\sum_{j \in \delta(S)} x_{Sj} = 1 \label{MILPdeg1}\\
    & \textstyle\sum_{j \in \delta(T)} x_{jT} = 1 \label{MILPdeg2}\\
    & \textstyle\sum_{j \in \delta(i) }x_{ij} -  \textstyle\sum_{j \in \delta(i) }x_{ji} = 0  && \forall i \in N\setminus \{S,T\} \label{MILPdeg3}\\
    &B_{max} \geq b_j \geq B_{min} && \forall j \in N \setminus \{S\} \label{MILPbatt1}\\ 
    &b_S = B_0 \label{MILPbatt2}\\
    &b_j \leq b_i - C_{ij}  + Z_{ij} g_{ij} \nonumber \\
    & \nonumber- V(1- \textstyle\sum_{k \in \delta(i) \setminus \{j\}} g_{ki}) \\  & + M(1 - x_{ij}) && \forall (i,j) \in E \label{MILPbatt3}\\
    &b_j \geq b_i - C_{ij}  + Z_{ij} g_{ij} \nonumber \\
    & \nonumber- V(1- \textstyle\sum_{k \in \delta(i) \setminus \{j\}} g_{ki}) \\  & - M(1 - x_{ij}) && \forall (i,j) \in E \label{MILPbatt4}\\
    &q_j \geq 0 && \forall i \in N \setminus \{S\} \label{MILPgen1}\\ 
    &q_S = Q_0 \label{MILPgen2}\\ 
    &q_j \leq q_i - Z_{ij}g_{ij} + M(1- x_{ij}) && \forall (i,j) \in E \label{MILPgen3}\\
    &g_{ij} \leq x_{ij} && \forall (i,j) \in E \label{MILPmisc1}\\
    &g_{ij} \leq G_{ij} && \forall (i,j) \in E \label{MILPmisc2} 
\end{align}

Equations \eqref{MILPobj}-\eqref{MILPmisc2} constitute the complete MILP definition of the NRHFSPP. Equation \eqref{MILPobj} gives the objective which is to minimize path cost. Equations \eqref{MILPdeg1}-\eqref{MILPdeg3} give degree constraints such that there is exactly one edge leaving the source node, one edge entering the goal node, and even degree for all other nodes. Equation \eqref{MILPbatt1} constrains the battery level to be between the maximum and minimum charge values. The battery state is updated between nodes as given in Equations \eqref{MILPbatt3} and \eqref{MILPbatt4}, where $C_{ij}$ is the drain on the battery to travel along edge $(i,j)$, $Z_{ij}$ recharges the battery if the generator is used, determined by $g_{ij}$. This constraint, using the big-\textit{M} method, is trivially satisfied if the edge $(i,j)$ is not used in the solution. The generator update equations are given in equations \eqref{MILPgen1}-\eqref{MILPgen3} and are formulated similar to the battery update equations using the same big-\textit{M} technique. The generator schedule is restricted using equations \eqref{MILPmisc1} and \eqref{MILPmisc2}; these equations ensure that the generator cannot be run on a noise restricted edge or an the edge that is not used in the solution.

A common  approach to solve a MILP is branch-and-bound \cite{mitchell2002branch} methods. This approached is broadly applicable to all MILPs and available through commercial solvers such as CPLEX or Gurobi. However, an alternative algorithm can often be implemented which exploits the specific problem structure and out-performs a branch-and-bound method with respect to computation speed. To this end, we develop a labeling algorithm, presented in Section \ref{sec:labeling}, as a fast alternative to solving MILP. 

This labeling algorithm makes use of a lower bound to improve computation time. For a given instance of the NRHFSPP, the Shortest Unconstrained Path (SUP) is the shortest path through the graph from start to goal node with no resource or noise constraints added.  This can easily be found by any standard SPP algorithm such as Djikstra's or $A^*$.

\begin{lemma}[]\label{proof:LB}
The optimal solution of the SUP, defined by Equations \eqref{MILPobj}-\eqref{MILPdeg3}, is a lower bound to the Noise Restricted Hybrid Fuel Shortest Path Problem as described in Equations \eqref{MILPobj}-\eqref{MILPmisc2}.
\end{lemma}

\begin{proof}
For a given NRHFSPP problem, the set of all feasible solutions is given by set $A$, which is the set of all solutions which satisfy Equations \eqref{MILPdeg1}-\eqref{MILPmisc2}. We refer to the solution with minimum cost from the set $A$ as $x_a^*$. For the same graph, start and goal nodes, and edge costs, we refer to the set of all feasible unconstrained paths by the set $B$. The set $B$ consists of all the paths that satisfy Equations \eqref{MILPdeg1}-\eqref{MILPdeg3},which is the feasible set of the shortest path problem. The optimal solution from set $B$ is referred to as $x_b^*$, and referred to here as the SUP. By definition, $B \subseteq A$, and thus $J(x_b^*) \leq J(x_a^*)$. 
\end{proof}
The effect of tight lower-bound on the computational time of the labeling algorithm is analyzed in Section \ref{sec:results}.

\begin{lemma}[]\label{proof:nphard} 
The Noise-Restricted Hybrid-Fuel UAV Shortest Path Problem, described above, is NP-Hard.
\end{lemma}
\begin{proof}
NP-hardness of the NRHFSPP is proved by showing that RCSPP, which is NP-complete \cite{handler1980dual}, is a special case of the NRHFSPP. The resource constraints in the RCSPP are equivalent to the battery constraints of the NRHFSPP if initial fuel level is 0.  In the NRHFSPP, a 0 fuel level means the generator cannot be run, making all noise-restrictions obsolete, and means the battery level is monotonically decreasing.  Thus, an NRHFSPP instance with initial fuel level at 0 is an instance of an RCSPP, and vice-versa. Therfore, an RCSPP can be solved as a special instance of the NRHFSPP, using any algorithm used to solve a standard NRHFSPP.  Therefore, knowing the RCSPP is NP-complete, the NRHFSPP is at least as hard as an NP-complete problem and thus is NP-hard.
\end{proof}
\section{Labeling Algorithm}\label{sec:labeling}
Here we develop and present a labeling algorithm to solve the NRHFSPP.  A general framework of labeling algorithms can be found in \cite{desrosiers1995time}. The algorithm presented here is similar in structure to the $A^*$ algorithm.  However, there exist fundamental differences to solve the problems where resources need to be tracked. Standard SPP algorithms such as Djikstra's and $A^*$ work by keeping an open list of labels, a single one for each node, to be expanded. Each label in this open list contains only one value, which in the case of Djikstra's is simply the cost of the best path found so far to that node, and each label as an associated path to reach that node with the given cost. Whenever an alternate path to a node is found with lower cost, the existing label is replaced with the label for the new path. Thus, only the lowest cost path to the each node needs to be saved. 

While solving path planning problems such as RCSPP, RCSPP-R, and NRHFSPP, resources need to be tracked along with the cost. To account for this, the definition of labels and the discarding of suboptimal labels differ significantly from SPP algorithms.  In case of resource-constrained path planning, both path cost and resource values must be tracked together while exploring the nodes, and therefore a label needs to be embedded with this information.  Furthermore, paths cannot be compared only in terms of path cost. Given some path to a node, and an alternate path of worse path cost but with more favorable resource consumption, the lower cost path cannot be sure to dominate the other.  The path of higher cost but more favourable resource consumption cannot be discarded as it may be the case that the path of lower cost becomes infeasible when extending it further where the higher cost path remains feasible due to the improved resource consumption.  However, a path can dominate another fully if it has improved cost \textit{and} improved resource consumption (across all resrouce values).  Thus, these resource-constrained path planing algorithms keep a \textit{list} of \textit{undominated} or \textit{efficient} labels for each node in the graph, where each label corresponds to a partial path to reach the associated node.


For NRHFSPP, a label will be of the form: $(d_i^k, b_i^k, q_i^k)$, this is the $k^{th}$ label for node $i$. The path and generator schedule from $S$ to $i$ associated with this label results in a cost of $d_i^k$, a battery charge of $b_i^k$, and a generator fuel state of $q_i^k$. A label $(d_i^1, b_i^1, q_i^1)$, \textit{dominates} another label $(d_i^2, b_i^2, q_i^2)$ if: $d_i^1 < d_i^2$ and $b_i^1 >  b_i^2$ and $q_i^1 > q_i^2$. A label is defined as \textit{efficient} relative to a group of other labels if no other label dominates it.  Here, labels can be compared lexicographically, so that inefficient labels can still be pruned to help reduce the search space. Dominated labels for a given node $i$ are discarded such that only efficient labels are considered in the search space.   

Labels for a given node may be equivalent in label values, although their corresponding path $X$ or generator pattern $Y$ may differ. Two labels $(d_i^1, b_i^1, q_i^1)$ and $(d_i^2, b_i^2, q_i^2)$ are equivalent if:  $d_i^1 = d_i^2$ and $b_i^1 =  b_i^2$ and $q_i^1 = q_i^2$. When such  equivalent labels occur in the search, one of those is arbitrarily discarded. Discarding equivalent labels allows to reduce the search space without losing optimality, and this could impact the computational effort especially in problems with high symmetry. 

We present the pseudo-code of the labeling algorithm in Algorithm \ref{algo:generic}.  The first selection method is referred to as the {\fontfamily{qcr}\selectfont LABEL} method given in Algorithm \ref{algo:LABEL_selection}, and returns the single label of minimum $f$-cost (defined shortly) in the open list.  The second selection method, referred to as {\fontfamily{qcr}\selectfont NODE}, is given in Algorithm \ref{algo:NODE_selection}, and finds the label of minimum $f$-cost in the open list and returns all labels in the open list for the node for which the minimum cost label is associated with. 

The algorithm first initializes an open and closed list as done in steps 1-2 of Algorithm \ref{algo:generic}.  The algorithm runs until the open list is empty, shown in step 3. In each iteration, the next labels to be treated are selected in step 4; either Algorithm \ref{algo:LABEL_selection} or \ref{algo:NODE_selection} can be used in place of $NEW\_LABELS$ to obtain the next set of labels to treat. If the selected labels correspond to the goal node, then the best label in this list is returned as the optimal solution. Otherwise, the selected labels are then treated as shown in steps 9-25. Treatment of a label involves removing it from the open-list (step 24), addition to the closed list (step 25), and extending the associated partial path and generator pattern to the adjacent nodes in the graph, each with the generator on and with the generator off (steps 12-23). These newly created labels are then added to the open list only if they are feasible and efficient, shown in steps 16-17 and 21-22.   Note that if the {\fontfamily{qcr}\selectfont LABEL} selection method of Algorithm \ref{algo:LABEL_selection} is used, the $LABELS$ list in step 4 of Algorithm \ref{algo:generic} is a single label, where as it is a list if the {\fontfamily{qcr}\selectfont NODE} method of Algorithm \ref{algo:NODE_selection} is used.
 
The {\fontfamily{qcr}\selectfont NODE} method was shown to perform better than the {\fontfamily{qcr}\selectfont LABEL} method in case of RCSPP-R \cite{smith2012solving} and a network relay problem \cite{cabralPhD}. We use the two selection methods in the labeling algorithm we develop to solve the NRHFSPP, and empirically evaluate their performance. It should be noted that the {\fontfamily{qcr}\selectfont NODE} selection in \cite{cabralPhD} has improved computational complexity due to the fast merge algorithm found in \cite{brumbaugh1989empirical}, which is applicable to 2-dimensional merge operations (cost, signal strength) but cannot be trivially extended to 3-dimensional merge operations as in our case (cost, battery, and fuel states). Given two lists $X$ and $Y$, \cite{brumbaugh1989empirical} merges them in $O(|X| + |Y|)$ for case of 2-dimensional list entries, whereas a less efficient approach as done for our 3-dimensional merge operations have running time $O(|X| \cdot |Y|)$.

The algorithm utilizes an estimated total cost.  The cost of the $k^{th}$ label for node $i$ is defined as $f_k(i) = d_i^k + h(i)$. The value $d_i^k$ is the cost to reach node $i$ from $S$ by the path of label $k$. The value of $h(i)$ is a heuristic estimate of the cost-to-go from node $i$ to $T$ which never over-estimates the true minimum cost-to-go. This estimated cost allows labels along shorter paths-to-goal to be preferentially selected over partial paths of similar $d_i^k$ but higher estimated cost-to-go. This also allows the stopping criterion that if a label for goal node $T$ is selected for treatment, this label can be returned as the optimal solution. In general, this greatly improves performance over a stopping criterion of searching until the open list $H$ is empty.

The effect of tightness of cost-to-go on algorithm performance is shown in Section \ref{sec:results}, and proof of exactness while using the cost-to-go is given in Section \ref{sec:exactness}. The running time of the algorithm is derived in Lemma \ref{proof:running_time}. 

\begin{algorithm}[htb]
\caption{Generic Labeling Algorithm}
\label{algo:generic}
\small
\SetKwInOut{Input}{input}
\SetKwInOut{Output}{output}
\Input{Cost Matrix: $D \in \R^{NxN}$ \\ Noise Restriction Matrix $G: \in \R^{NxN}$ \\ Adjacency Matrix $A: \in \R^{NxN}$\\ Energy Cost Matrix $C: \in \R^{NxN}$ \\ Energy Transfer Matrix $Z: \in \R^{NxN}$ \\ Initial Battery/Generator States: $B_0, Q_0$\\Starting node index: $S$\\ Final node index: $T$ \\
}
\Output{Minimal Cost Path $X^*$ and Generator Pattern $Y^*$}        
\SetAlgoLined
 $H \gets \lbrace (0, B_0, Q_0) \rbrace$ // \CommentSty{Open list, add start node label}\\ 
 $P \gets \varnothing $ // \CommentSty{ Closed set}\\ 
\While{$H \neq \varnothing$}{ 
    \texttt{LABELS}, $i$ = \texttt{NEW\_LABELS}($H$) // \CommentSty{this calls the functions \texttt{NODE} or \texttt{LABEL}}\\
    \uIf{$i$ = T}{
        $l^* \gets \argmin_{l \in \texttt{LABELS}} f(l)$ \\
         $X^*$, $Y^*$ $\gets$ \texttt{EXTRACTPATH}($F^*$) // \CommentSty{returns the path corresponding to the label $F^*$} \\
        \Return $X^*$, $Y^*$
    }
    \For{$l_k \in \texttt{LABELS}\, \& \, l_k \notin P$}{ 
        // \textit{$l_k: \hspace{1cm} (d_i^k, b_i^k, q_i^k)$ for node $i$ and associated $X_k$, $Y_k$} \\
        
        $X_k$, $Y_k$ $\gets$ \texttt{EXTRACTPATH}($l_k$) \\
        \For{$j \in $ \texttt{NEIGHBORS}($i$)}{
            // Extend $l_k$ with gen on: \\
            $X_a \gets [X_k, j]$,  $Y_a \gets [Y_k, 1]$ \\
            $F_A \gets (d_i^k  + D_{ij}, b_i^k - C_{ij} + Z_{ij}, q_i^k - Z_{ij})$ \\
            \uIf{\texttt{FEASIBLE}($F_A$) \& $j \not\in X_k$ \& \texttt{EFFICIENT}(H, $F_A$)}{
                $H \gets H \cup \hspace{3cm} F_A$ \\}
                // Extend $l_k$ with gen off: \\
                $X_b \gets [X_k, j]$,  $Y_b \gets [Y_k, 0]$ \\
                $F_B \gets (d_i^k  + D_{ij}, b_i^k - C_{ij}, q_i^k)$ \\
            \uIf{\texttt{FEASIBLE}($F_B$) \& $j \not\in X_k$ \& \texttt{EFFICIENT}(H, $F_B$)}{
                $H \gets H \cup \hspace{3cm} F_B$ \\}
        } 
        $H$ $\gets$ $H \setminus l_k$ \\ 
        $P$ $\gets$ $P$ $\cup \hspace{3cm}$ $l_k$ \\
    }
}
\end{algorithm}

\begin{algorithm}[htb]
\small
        \SetKwInOut{Input}{input}
        \SetKwInOut{Output}{output}
        
        \Input{Open Set: $H$}
        \Output{Label with minimum f-cost, and corresponding node}

\SetAlgoLined
 $F \gets \argmin_{l \in H} f(l)$ \\
 
  \Return $F$, \texttt{NODE\_OF\_LABEL}($F$)
\caption{{\fontfamily{qcr}\selectfont LABEL}  Selection Method} \label{algo:LABEL_selection}
\end{algorithm}

\begin{algorithm}[htb]
\small
        \SetKwInOut{Input}{input}
        \SetKwInOut{Output}{output}
        
        \Input{Open Set: $H$}
        \Output{Set of new labels to be treated}

\SetAlgoLined
 $F \gets \argmin_{l \in H} f(l)$ \\
 $n \gets  \texttt{NODE\_OF\_LABEL}  (F) $ \\

  \Return \texttt{LABELS\_OF\_NODE}($n$), $n$
\caption{{\fontfamily{qcr}\selectfont NODE}  Selection Method} \label{algo:NODE_selection}
\end{algorithm}

\subsection{Optimality}\label{sec:exactness}
The labeling algorithm presented is exact and it is proved by the two lemmas stated below.
\begin{lemma}[]\label{proof:never_prune_optimal}
In Algorithm \ref{algo:generic}, labels corresponding to the optimal path are never pruned.
\end{lemma}
\begin{proof}
Let $P^* = \{S, \ldots, T\}$ be the optimal path from start node $S$ to goal node $T$, and let $p_i^*$ be a sub-path of $P^*$ from node $S$ to node $i$, $\forall i \in P^*$. The sub-path $p_i^*$ has associated label $(d_i^{*}, b_i^{*}, q_i^{*})$. When this label is checked for efficiency in the open list, $H_i$, it will never be pruned as it is the optimal partial path to node $i$. No other label corresponding to $i$ that is a subpath of a feasible $S-T$ path can have a label cost smaller than $D_i^{*}$. If an alternate label is of equal label values such that $d_i^* = d_i^k$, $b_i^* = b_i^k$, and $q_i^* = q_i^k$, then one of them is arbitrarily added and the optimal is not pruned as the labels are equivalent in all metrics. If there exist labels with equal cost but different values of $b_i$ and $q_i$, then only the dominated labels are pruned. Because the labels corresponding to the minimum cost path will remain un-dominated, they are never pruned. This is true for all nodes in $N$, therefore labels corresponding to the optimal path are never pruned at any node $i \in P^*$.
\end{proof}

By Lemma \ref{proof:never_prune_optimal} it is known that labels corresponding to the optimal path will never be pruned. In Lemma \ref{proof:never_pick_suboptimal} it is proven that the labels for the optimal feasible path will always be selected for expansion  before feasible, suboptimal labels of greater cost.

\begin{lemma}[]\label{proof:never_pick_suboptimal}
The labeling algorithm  described in Algorithm \ref{algo:generic} using label selection method given in Algorithm \ref{algo:LABEL_selection} will never select a label to expand corresponding to a partial path of a feasible, suboptimal path.
\end{lemma}
\begin{proof}
 Consider three labels: i) $A$, a label for the goal node $T$ which is sub optimal; ii) $B$, a label for the goal node $T$ which is optimal; iii) $C$, a label for node $j$, which corresponds to the sub-path from node S to node $j$ of the path associated with label $B$. 

Let $g(Z)$ be the cost-to-arrive for the path associated with a label $Z$. The total cost of label $Z$ is defined by $f(Z) = g(Z) + h(i)$ where $i$ is the node associated with label $Z$, and $h(i)$ is a heuristic estimate of the cost-to-go and is a lower bound to the optimal path from node $i$ to goal node $T$. Note that $h(T) = 0$ and thus $f(A) = g(A)$ and $f(B) = g(B)$.

By definition, $g(A) > g(B)$. Assume the labeling algorithm selects label $A$ instead of $C$. If $h(\cdot)$ is an admissible heuristic, then $g(B) = f(B) \geq f(C)$, as minimum-cost-to-go is greater than or equal to estimated cost-to-go by definition of the admissible heuristic.

If $C$ is \textit{not} chosen for expansion over $A$, this implies $f(C) \geq f(A)$. Therefore, $f(B) \geq f(C) \geq f(A) \implies  f(B) \geq f(A)$. 
%
This contradicts the assumption that $A$ is suboptimal, and proves the lemma.
\end{proof}

Labels corresponding to the optimal solution will never be pruned, and labels for the optimal solution will always be treated over labels corresponding to a feasible, sub-optimal solution. The above two lemmas together complete the proof of exactness of Algorithm \ref{algo:generic} when using label selection \ref{algo:LABEL_selection}. 

 \subsection{Expected Running Time}\label{subsec:running_time}
 The pseudo-polynomial running time is derived under the assumption that resource consumption and regeneration take only integer values. That is, $Z_{i,j}, C_{i,j} \in \mathbb{Z} \quad \forall (i,j) \in E$. With such a restriction, the number of labels per node can be bound based on maximum battery charge and fuel level. This is done to find the worst-case running time, where the worst-case for the algorithm is finding and treating every possible label in the problem, thus giving an upper bound on running time. 

 \begin{lemma}[]\label{proof:running_time}
 The Labeling algorithm described in Algorithm \ref{algo:generic} when using selection method of Algorithm \ref{algo:LABEL_selection} has a pseudo-polynomial running time.    
 \end{lemma}
 \begin{proof}

We first define the maximum number of labels for node $i$ as $D^i = \max (B_{max}, Q^i_{max})$ where $Q^i_{max}$ is the maximum possible fuel level at node $i$. This can be found as $Q^i_{max} = \min(Q_0 - (E^*_i - B_0), Q_0)$, where $E^*_i$ is the minimum energy cost to reach node $i$. Calculating $Q^i_{max}$ rather than using $Q_0$ accounts for nodes which   must use some level of generator fuel to reach. This minimum fuel usage is calculated by $(E^*_i - B_0)$, being a positive number only if the minimum energy path uses more energy than the initial battery state. If this is not the case, $Q_0 < Q_0 - ((E^*_i - B_0))$ and therefore $Q^i_{max} = Q_0$. 

Next we define $D = \sum_{i \in N} D^i$ as the total labels on the graph and $d = max_{i \in N} D^i$ as the largest number of labels for a node across the whole graph. We also define $e_{max} = \max_{i \in N} \delta(i)$ as the greatest number of edges out of a node across all nodes, where $\delta(\cdot)$ is the number of edges out of node $i$. Note that $e_{max} \leq N$.

\noindent We simplify the pseudo-code first for conciseness:

i) Initialize open and closed list - O(N)

ii) Pick label - O(1) 

iii) Treat current label - O(2 $e_{max} d$):

iv)  For each outgoing arc, make new label q

iv) Return to step ii)

We can bound the number of label treatments based on the above simplified form. Step 2 occurs at most $D$ number of times, and will take O(1) operations if a heap is used. For each occurrence of step 2, step 3 is executed. Step 3 will occur at most $D$ times, once for every occurrence of step 2. The complexity of step 3 can be bounded based on the number of edges out of the nodes. When treating a label, the maximum number of edges out will be $e_{max}$, and 2 new labels will be created for each edge during treatment (gen on and gen off). Thus, at most $2 e_{max}$ new labels will be made. For each new label, it must be compared to prior labels made for the given node, which is at most $d$ number of labels, thus each comparison for each new label is at most $d$ operations. Therefore, the operations for step 3 is bounded by $2 e_{max} d$. The running time for all steps is therefore: $O(N + D(1 + 2 e_{max} d )) \approx O(Dd e_{max})$.
\end{proof} 

Note that the preceding proof holds generally for the case when resource consumption and regeneration ($Z_{i,j}$ and $C_{i,j}$) take on non-integer quantized values, i.e., the set of values for the resource consumption and regeneration is a finite set.

\begin{figure}[t]
    \centering
    \includegraphics[width = 0.4\textwidth]{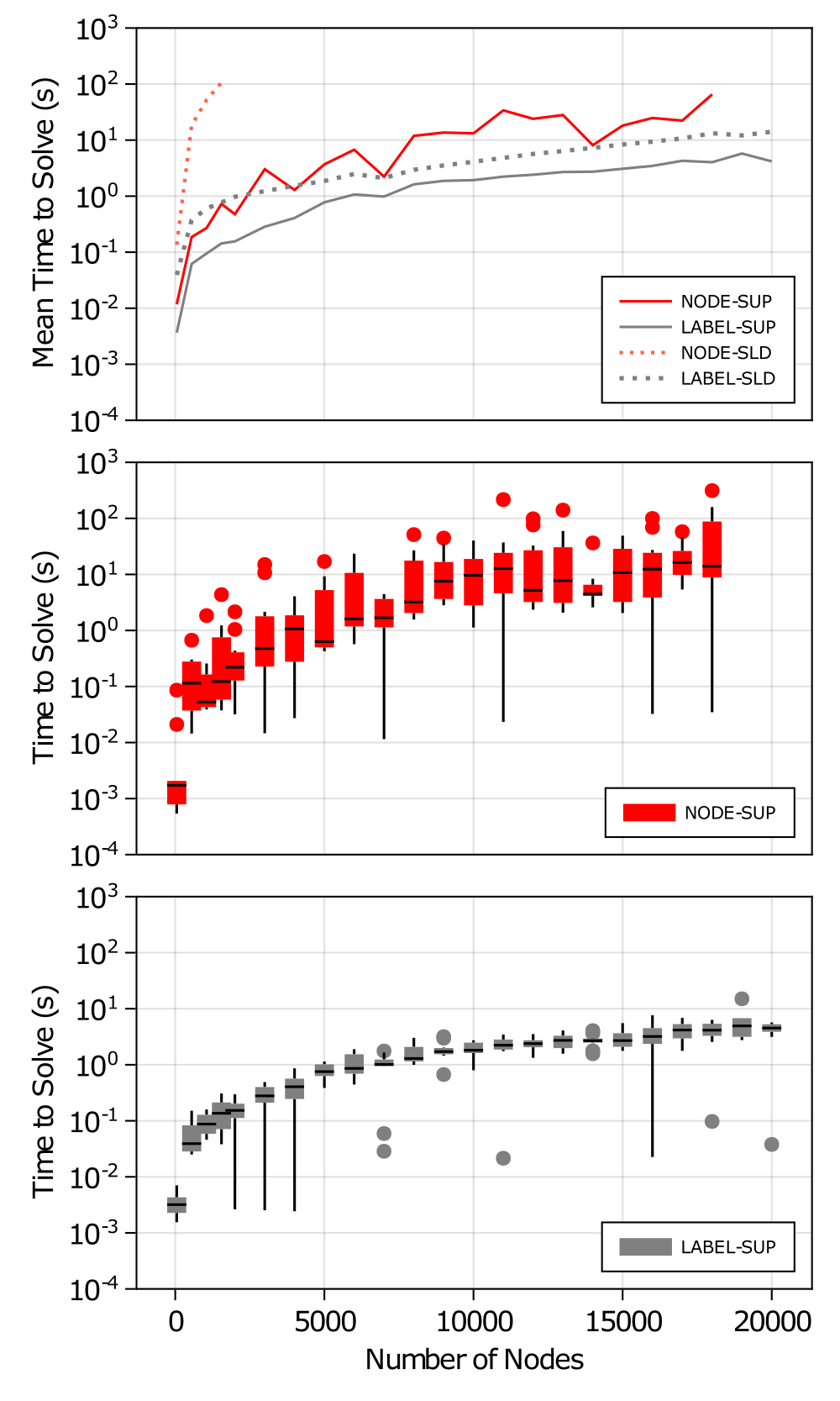}
    \caption{Computational Times for 2D Euclidean Graph versus Number of Nodes - i) Top: Results for SUP and SLD Lower Bounds (LB) for NODE and LABEL Selection Methods; ii) Middle: Boxplot for NODE-SUP LB using NODE Selection Algorithm; iii) Middle: Boxplot for NODE-SLD LB using LABEL Selection Algorithm }
    \label{fig:euc2D}
\end{figure}
\begin{figure}[t]
    \centering
    \includegraphics[width = 0.4\textwidth]{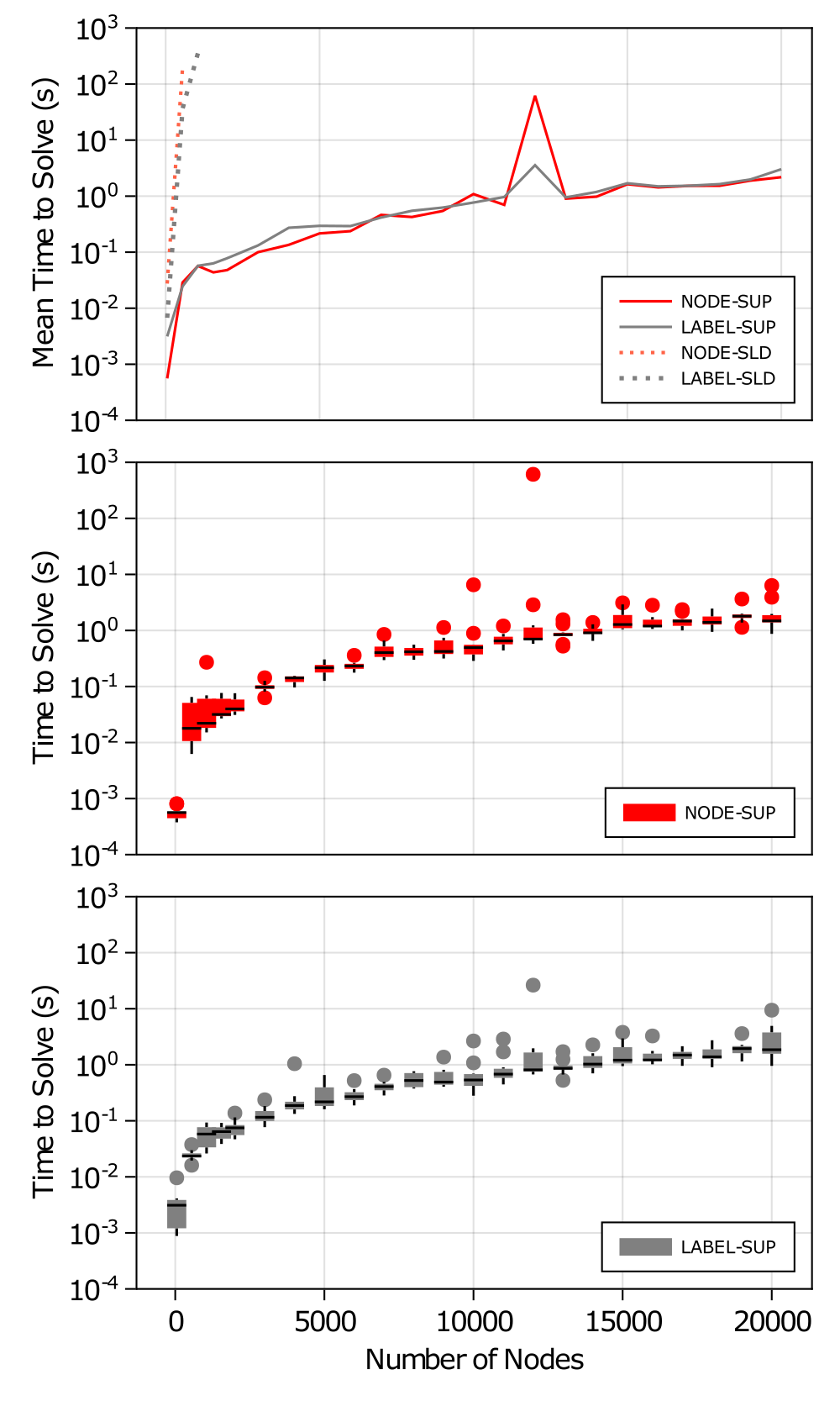}
    \caption{Computational Times for 3D Euclidean Graph versus Number of Nodes - i) Top: Results for SUP and SLD Lower Bounds (LB) for NODE and LABEL Selection Methods; ii) Middle: Boxplot for NODE-SUP LB using NODE Selection Algorithm; iii) Middle: Boxplot for NODE-SLD LB using LABEL Selection Algorithm }
    \label{fig:euc3D}
\end{figure}

\begin{figure}
    \centering
    \includegraphics[width = 0.4\textwidth]{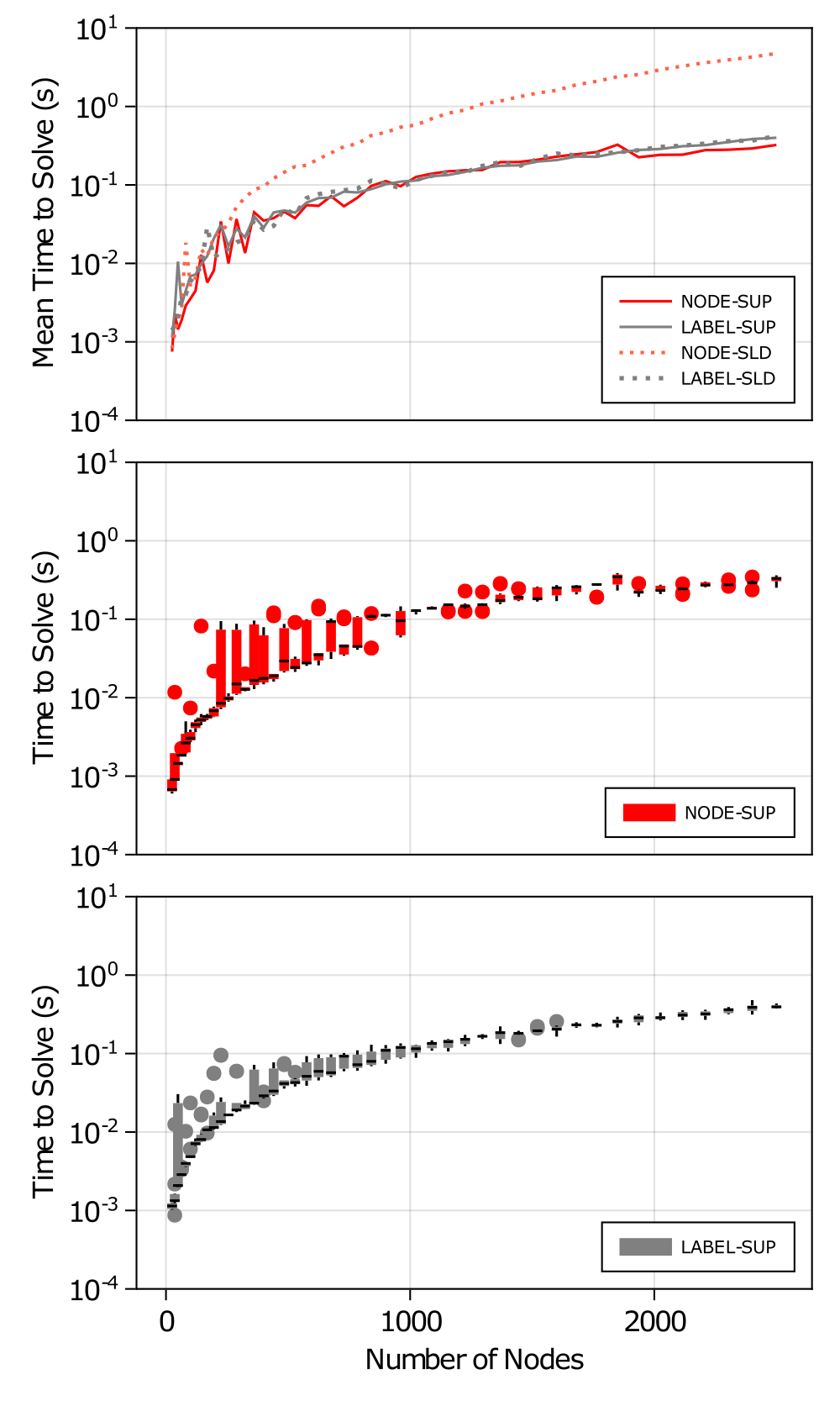}
    \caption{Computational Times for 2D Lattice Graph versus Number of Nodes - i) Top: Results for SUP and SLD Lower Bounds (LB) for NODE and LABEL Selection Methods; ii) Middle: Boxplot for NODE-SUP LB using NODE Selection Algorithm; iii) Middle: Boxplot for NODE-SLD LB using LABEL Selection Algorithm }
    \label{fig:lattice2D}
\end{figure}

\begin{figure}[t]
    \centering
    \includegraphics[width = 0.4\textwidth]{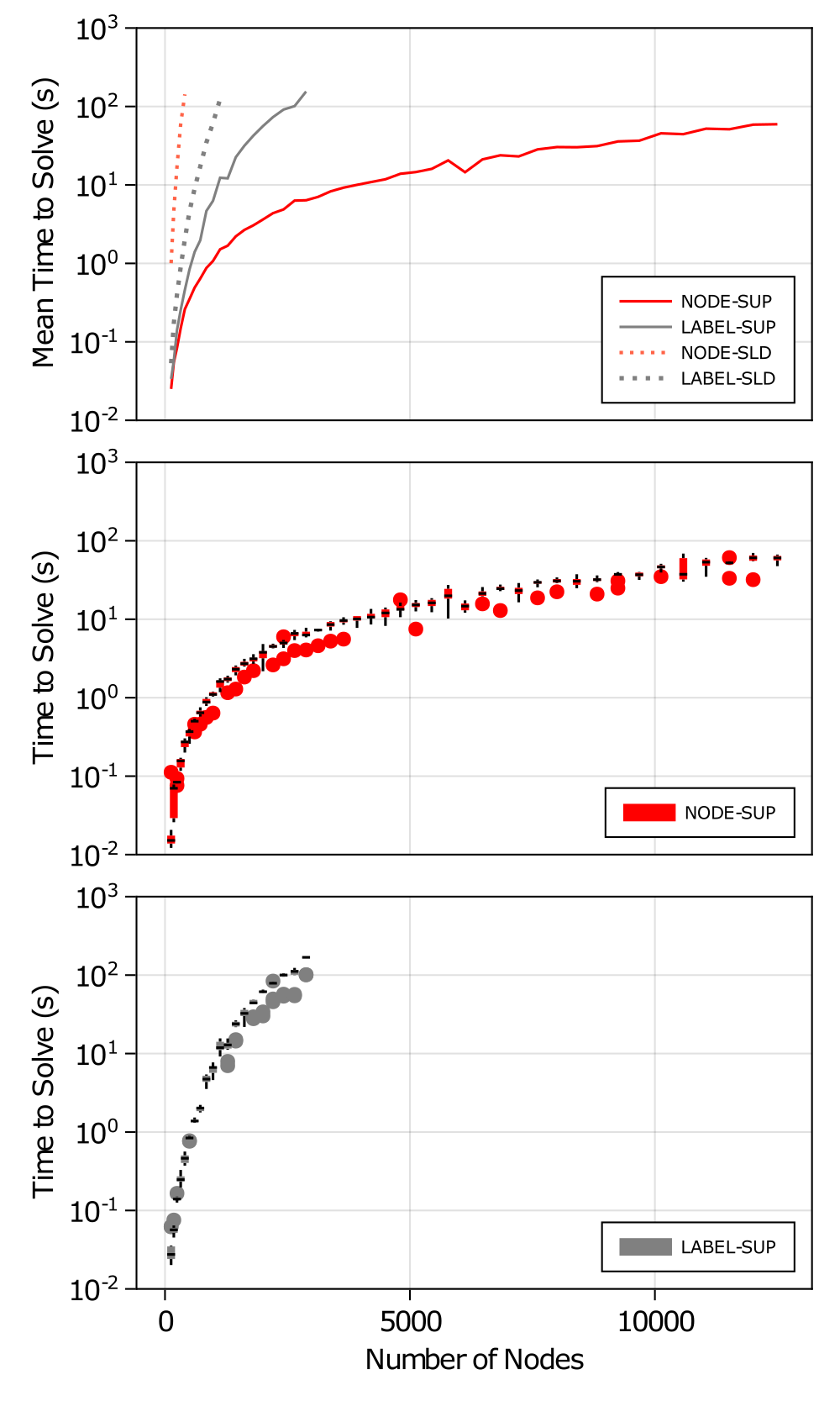}
    \caption{Computational Times for 3D Lattice Graph versus Number of Nodes - i) Top: Results for SUP and SLD Lower Bounds (LB) for NODE and LABEL Selection Methods; ii) Middle: Boxplot for NODE-SUP LB using NODE Selection Algorithm; iii) Middle: Boxplot for NODE-SLD LB using LABEL Selection Algorithm }
    \label{fig:lattice3D}
\end{figure}
\begin{figure}
    \centering
    \includegraphics[width = 0.4\textwidth]{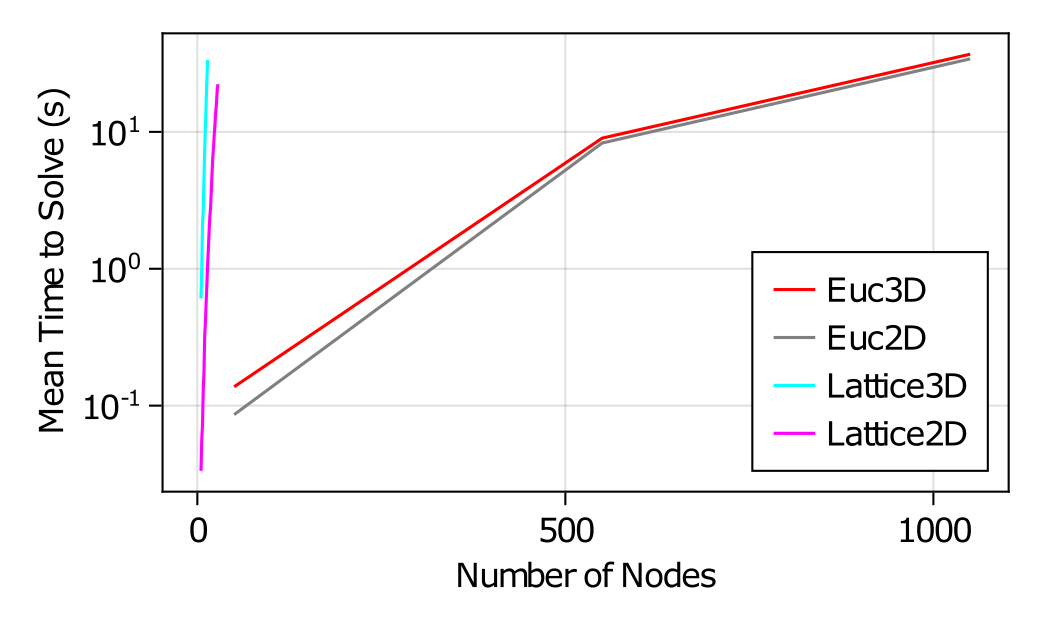}
    \caption{CPLEX Branch and Bound Results}
    \label{fig:CPLEX}
\end{figure} 

\section{Results}\label{sec:results}
We present the results of computational experimental with two different types of  graphs: lattice graphs and Euclidean graphs, samples of each are shown in Figures \ref{fig:graph_euc} and \ref{fig:graph_lattice}. In lattice graphs, nodes are evenly placed in a lattice/grid, and noise restricted zones are randomly generated as rectangles (or rectangular prisms for 3D case) in the airspace such that edges within these regions are noise-restricted. The other set of experiments are done using Euclidean graphs that are produced by randomly placing node locations, with each node connected to its 4 nearest neighbors by Euclidean distance. The generator noise restrictions are chosen randomly similar to the lattice graphs. The average percentage of noise-restricted edges over all instances is $32$\%. For each graph type and size, 10 instances were generated to be solved by the labeling algorithm with both the node selection and label selection methods.

For all problem instances, the edge cost $D_{ij}$ is the euclidean or straight-line-distance (SLD) between nodes $i$ and $j$.  Further, the energy consumption $C_{ij}$ and generator energy transfer $Z_{ij}$ along an edge are proportional to the edge cost, and are the same in either direction along the edge, with exception of edges with changes in altitude. Traveling along edges with an increase in altitude has higher battery drain $C_{ij}$ and generator energy transfer $Z_{ij}$. Edges with a decrease in altitude allow gliding through the edge such that there is no battery drain, and proportionally lower energy transfer $Z_{ij}$.

For each instance, the start and end goals are chosen to be the two nodes farthest apart (by Euclidean distance). In general, start and goal nodes which are closer together are easier problems to solve. Placing them as far apart as possible ensures that for a given graph size, the most difficult instances are generated. This is to the end of mitigating variance in difficulty between instances of same graph size due variability of start-goal distances.

The computational experiments were done for each graph type in both 2D and 3D environments. In 3D environments, the Euclidean graphs are produced similar to the 2D instances, where nodes are randomly placed in 3D space. The 3D lattice graphs have nodes placed in a 3-dimensional lattice. Noise restrictions are generated in a similar manner, where prisms are randomly generated and all the edges with both ends inside the prisms are set to be noise restricted. 

The Euclidean graphs connect each node to its 4 closest neighbors, such that each node has a minimum of $4$ edges connected to it with some nodes having more than $4$. The 2D lattice graphs, by definition, have $4$ outgoing edges from each node, except for the boundary nodes. The 3D lattice graphs have $12$ edges out of each node with exception to the nodes at the limit of the space.  Here, $4$ orthogonal movements (left, right, forward, backward) are allowed in addition edges for the $4$ orthogonal movements while moving up or down.  

The exactness of the labeling algorithm is experimentally verified by solving the MILP formulation for a subset of the problem instances with CPLEX \cite{cplex2009v12} solver. This verification was done only for the instances where computation time of CPLEX is tractable. The time-to-solve with CPLEX is shown in Figure \ref{fig:CPLEX}, where it is seen that the problems quickly becomes intractable relative to the scaling of the labeling algorithm.

The computation time required by the labeling algorithm using each of the two label selection methods ({\fontfamily{qcr}\selectfont NODE} and {\fontfamily{qcr}\selectfont LABEL}) with respect to graph size (number of nodes) is shown in Figures \ref{fig:euc2D} and \ref{fig:euc3D} for the Euclidean graphs, and in Figures \ref{fig:lattice2D} and \ref{fig:lattice3D} for the lattice graphs. The plots show results when SUP and Straight-Line-Distance (SLD) are used as Lower Bound (LB). For the Euclidean Graphs, it is clear that the {\fontfamily{qcr}\selectfont NODE} and {\fontfamily{qcr}\selectfont LABEL} selection methods both scale well. The performance of the two  is quite similar for the 3D Euclidean graphs. However, for the instances on the 2D Euclidean graphs, the {\fontfamily{qcr}\selectfont LABEL} selection method is clearly better, and it could solve instances with 20,000 nodes with only few seconds of computation time required. Whereas, the {\fontfamily{qcr}\selectfont NODE} method takes over a minute to solve those instances. In case of 2D lattice instances, the performance of the two selection methods is similar, the {\fontfamily{qcr}\selectfont NODE} method is little faster. The difference in the performance between the two selection methods is quite large for the 3D lattice graphs, where the {\fontfamily{qcr}\selectfont NODE} method performs significantly better than the {\fontfamily{qcr}\selectfont LABEL} method. Here, the {\fontfamily{qcr}\selectfont LABEL} takes few minutes to solve instances of 2880 nodes, and the {\fontfamily{qcr}\selectfont NODE} method solves those same instances in few seconds. This method was able to solve instances of up to 12500 nodes in less than a minute. As stated above, the 3D lattice instances have significantly higher connectivity between nodes, which in general makes the problems more difficult due to higher branching factor. The {\fontfamily{qcr}\selectfont NODE} method performs significantly better, indicating it is more robust to increases in graph connectivity as compared to the {\fontfamily{qcr}\selectfont LABEL} method. Overall, it is clear both methods are able to solve large instances of the NRHFSPP.

It should be noted that there is a larger variance in time-to-solve in the Euclidean graphs as compared to the lattice graphs. This is likely due to the uniform structure of the lattice graphs, where the only difference between graphs of the same size is the placement of quiet zones. The Euclidean graphs have higher variability in graph structure due to the random placement of nodes and the quiet zones, resulting in greater variance in time-to-solve when compared to the lattice graphs.

The cost-to-go estimation has a significant effect on performance of the labeling algorithm. As with the $A^*$ algorithm, improved tightness on cost-to-go for each node in the graph in general leads to improved computational time required to solve the problem. This effect is shown in Figures \ref{fig:euc2D} and \ref{fig:euc3D} for the Euclidean graphs and in Figures \ref{fig:lattice2D} and \ref{fig:lattice3D} for the Lattice graphs. The time-to-solve using SUP to estimate cost-to-go (calculated via $A^*$) is shown alongside the same using straight-line-distance (SLD) to the goal node to estimate cost-to-go. Based on the results, it is clear that a tighter estimate on cost-to-go improves performance for both the lattice and Euclidean graphs. This behavior is expected, as a tighter cost-to-go estimation reduces the number of labels created in a similar manner of cost-to-go tightness when using $A^*$ for SPP. 

\subsection{Effect of Graph Connectivity}
In this section, we study more directly how the graph connectivity affects time-to-solve. We consider only the 2D Euclidean graphs, as they saw the best overall performance, with edge parameters as described above.  
For a given node placement, we vary the number of edge connections to test variation in graph connectivity. This is illustrated in Figures \ref{fig:4conn} and \ref{fig:12conn}, for a given placement of 2000 nodes. Note that the edge counts shown in the figure captions are counting in an undirected manner, such that an edge between two nodes is counted once while allowing travel in either direction.
 
We test graphs from 50 nodes up to 20,000 nodes, with 10 instances generated for each category. For each instance (same node placement and noise-restricted zones), we add edge connections from 4 to 12 nearest-neighbor connections. 

\begin{figure}
     \centering
     \begin{subfigure}[b]{0.2\textwidth}
         \centering
         \includegraphics[width=\textwidth]{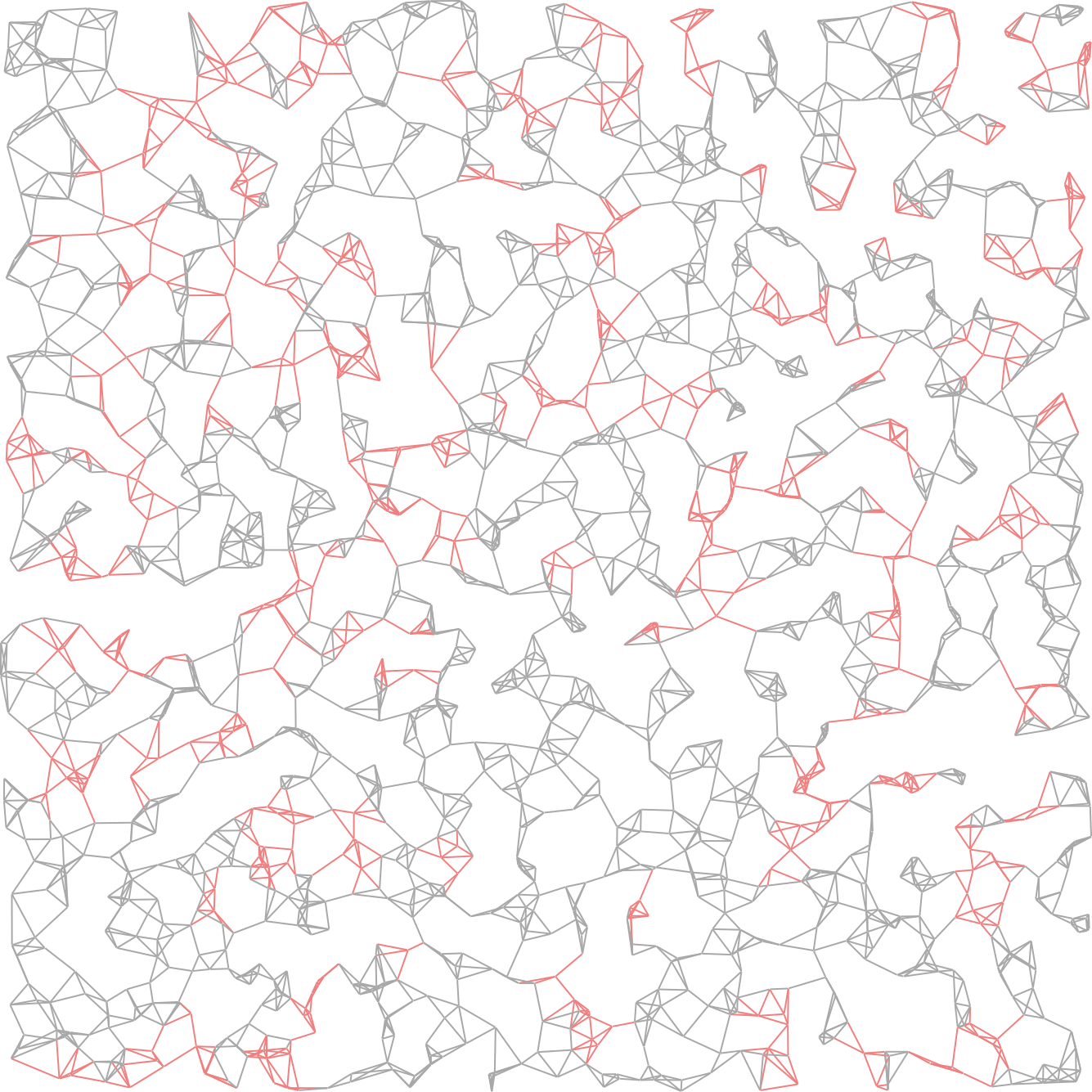}
         \caption{4-nearest-neighbor edge connections - 4869 total edges}
         \label{fig:4conn}
     \end{subfigure}
     \begin{subfigure}[b]{0.2\textwidth}
         \centering
         \includegraphics[width=\textwidth]{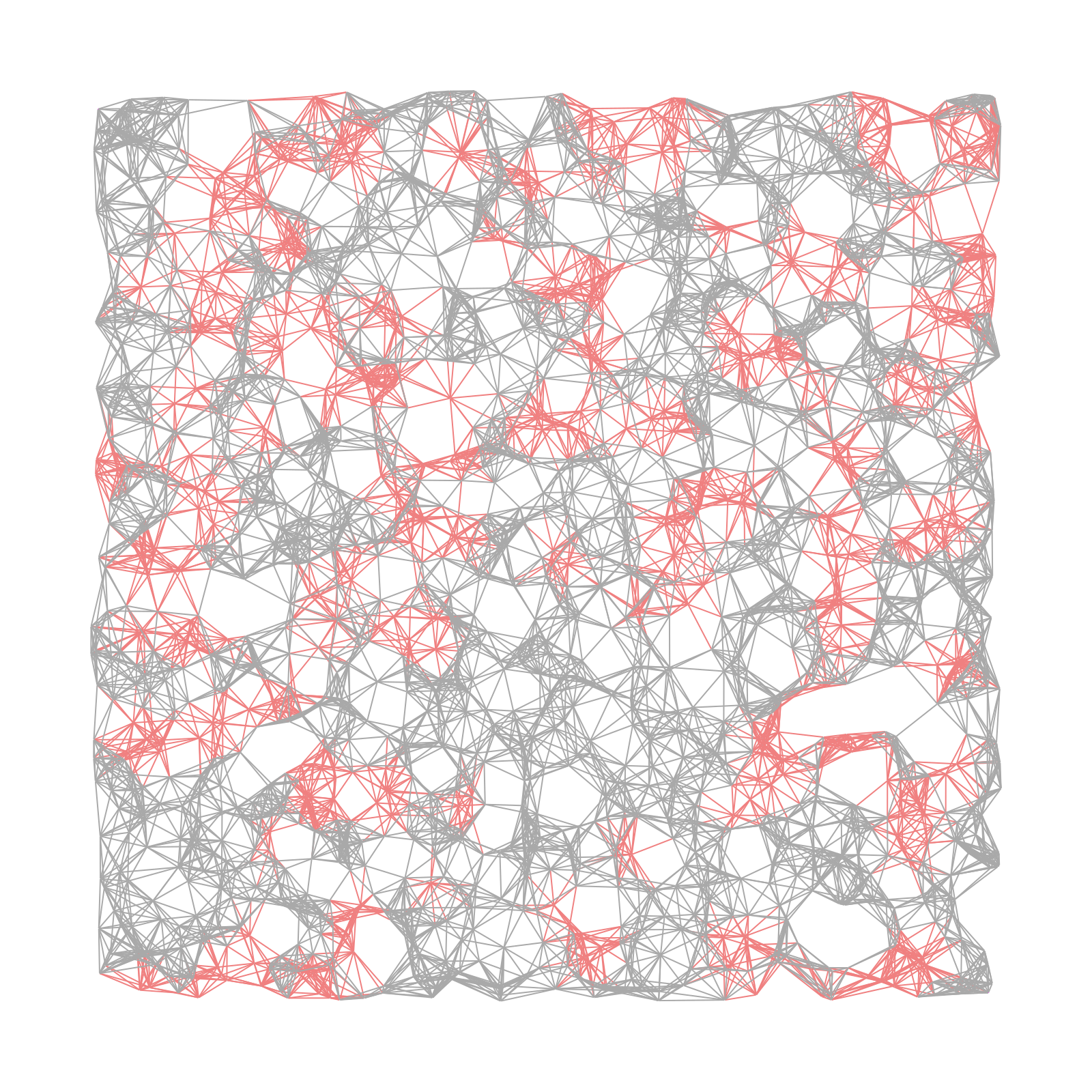}
         \caption{12-nearest-neighbor edge connections - 13644 total edges}
         \label{fig:12conn}
     \end{subfigure}
     \hfill
     \caption{Euclidean Graph with 2000 Nodes with Different Connectivity}
     \label{fig:three graphs}
\end{figure}

\begin{figure}
    \centering
    \includegraphics[width = 0.4\textwidth]{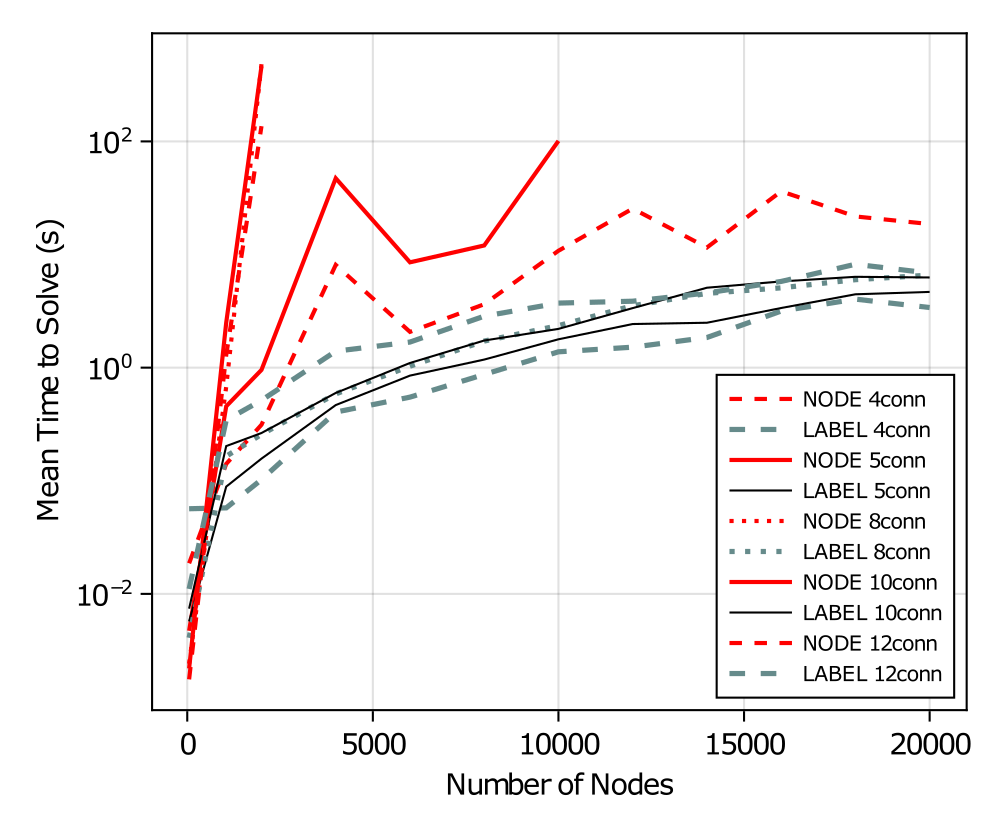}
    \caption{Computational Times for 2D Euclidean Graphs with Different Connectivity using the NODE and LABEL Selection Methods}
    \label{fig:connectivity}
\end{figure}

Figure \ref{fig:connectivity} shows the results of this testing. We see that the {\fontfamily{qcr}\selectfont LABEL} selection method scales much better with increases in connectivity than the {\fontfamily{qcr}\selectfont NODE}, where the latter quickly becomes intractable for instances of larger connectivity. The {\fontfamily{qcr}\selectfont LABEL} selection method shows an obvious increases in computation time as connectivity increases, though not as harsh as the change in {\fontfamily{qcr}\selectfont LABEL} computation time. Though the instances with higher connectivity are significantly harder to solve, the {\fontfamily{qcr}\selectfont LABEL} selection method still requires only a few seconds of computation time for instances with $20,000$ nodes. 

We infer that poor performance of the {\fontfamily{qcr}\selectfont NODE} algorithm relative to the {\fontfamily{qcr}\selectfont LABEL} is directly due to the increased connectivity. In case of the {\fontfamily{qcr}\selectfont NODE} selection, a higher connectivity means that for each label treated there is a larger number of new labels created.  In addition to more new labels for each lablel treated, this means a greater number of labels in the OPEN list for any given node.  The two compound together for a significant increase in total labels which are treated during algorithm runtime.  This significantly affects the performance of the {\fontfamily{qcr}\selectfont NODE} selection method.  Advantages of the {\fontfamily{qcr}\selectfont NODE} selection method are outweighed by this increase in total labels generated as connectivity increases. In the {\fontfamily{qcr}\selectfont LABEL} method, while total new labels created when treating a single label increases with the edge connectivity, the total number of labels is still lower than the {\fontfamily{qcr}\selectfont NODE} due to single label selection and treatment. This difference between the two methods in total number of labels produced grows as connectivity grows, as evidenced by increased runtime.

\section{Conclusions and Future Work}\label{sec:conc}
A Noise-Restricted Hybrid-Fuel Shortest Path Problem was presented which  arises in several urban path planning applications. A MILP formulation was given, and it was proven that the problem is NP-hard.  We presented a labeling algorithm  to solve the NRHFSPP, with two different selection methods to treat the labels. We prove that the complexity of the algorithm presented is pseudo-polynomial. Extensive computational experiments were conducted on different types of graph to analyze the algorithm. We present the performance of the algorithm with the two selection methods and show that the {\fontfamily{qcr}\selectfont LABEL} selection method is more robust to increases in graph connectivity.  In certain instances, specifically the 3D lattice problems the {\fontfamily{qcr}\selectfont NODE} selection method performs better. 

This work can be extended in a variety of ways. First is acceleration of the labeling algorithm, improving both the pseudo-polynomial worst-case running time and actual performance through alternate label selection or treatment processes. Alternatively, the single-agent problem can be explored in a continuous framework.  Similarly, a hybrid approach could leverage the discrete solution, presented in this study, as a starting point for the continuous-time paradigm. Further, multi-agent extensions of this problem could also be considered.

\bibliographystyle{IEEEtran}
\bibliography{shorttitles.bib, bibfile.bib}

\begin{thebibliography}{10}
\providecommand{\url}[1]{#1}
\csname url@samestyle\endcsname
\providecommand{\newblock}{\relax}
\providecommand{\bibinfo}[2]{#2}
\providecommand{\BIBentrySTDinterwordspacing}{\spaceskip=0pt\relax}
\providecommand{\BIBentryALTinterwordstretchfactor}{4}
\providecommand{\BIBentryALTinterwordspacing}{\spaceskip=\fontdimen2\font plus
\BIBentryALTinterwordstretchfactor\fontdimen3\font minus
  \fontdimen4\font\relax}
\providecommand{\BIBforeignlanguage}[2]{{%
\expandafter\ifx\csname l@#1\endcsname\relax
\typeout{** WARNING: IEEEtran.bst: No hyphenation pattern has been}%
\typeout{** loaded for the language `#1'. Using the pattern for}%
\typeout{** the default language instead.}%
\else
\language=\csname l@#1\endcsname
\fi
#2}}
\providecommand{\BIBdecl}{\relax}
\BIBdecl

\bibitem{TOWNSEND2020e05285}
A.~Townsend, I.~N. Jiya, C.~Martinson, D.~Bessarabov, and R.~Gouws, ``A
  comprehensive review of energy sources for unmanned aerial vehicles, their
  shortfalls and opportunities for improvements,'' \emph{Heliyon}, vol.~6,
  no.~11, p. e05285, 2020.

\bibitem{yedavalli2019assessment}
P.~Yedavalli and J.~Mooberry, ``{An assessment of public perception of urban
  air mobility (UAM)},'' \emph{Airbus UTM: Defining Future Skies}, 2019.

\bibitem{torija2020effects}
A.~J. Torija, Z.~Li, and R.~H. Self, ``Effects of a hovering unmanned aerial
  vehicle on urban soundscapes perception,'' \emph{Transp Res D Transp
  Environ}, vol.~78, 2020.

\bibitem{cussen2022uav}
K.~Cussen, S.~Garruccio, and J.~Kennedy, ``{UAV noise emission—A combined
  experimental and numerical assessment},'' in \emph{Acoustics}, vol.~4,
  no.~2.\hskip 1em plus 0.5em minus 0.4em\relax MDPI, 2022, pp. 297--312.

\bibitem{moshkov2021study}
P.~Moshkov, V.~Samokhin, and A.~Yakovlev, ``{Study of the noise sources of an
  UAV with a two-stroke engine and shrouded propeller},'' in \emph{J. Phys.
  Conf.}, vol. 1925, no.~1, 2021.

\bibitem{gasparetto2012trajectory}
A.~Gasparetto, P.~Boscariol, A.~Lanzutti, and R.~Vidoni, ``Trajectory planning
  in robotics,'' \emph{Math. Comput. Sci.}, vol.~6, pp. 269--279, 2012.

\bibitem{scholer2011configuration}
F.~Scholer, A.~la~Cour-Harbo, and M.~Bisgaard, ``Configuration space and
  visibility graph generation from geometric workspaces for {UAV}s,'' in
  \emph{AIAA Paper 2011-6416}.

\bibitem{fortune1995voronoi}
S.~Fortune, ``{Voronoi diagrams and Delaunay triangulations},'' \emph{Computing
  in Euclidean geometry}, pp. 225--265, 1995.

\bibitem{aurenhammer1991voronoi}
F.~Aurenhammer, ``Voronoi diagrams—a survey of a fundamental geometric data
  structure,'' \emph{ACM Computing Surveys (CSUR)}, vol.~23, no.~3, pp.
  345--405, 1991.

\bibitem{erten2009quality}
H.~Erten and A.~{\"U}ng{\"o}r, ``Quality triangulations with locally optimal
  steiner points,'' \emph{SIAM J Sci Comput}, vol.~31, no.~3, pp. 2103--2130,
  2009.

\bibitem{klesh2007energy}
A.~Klesh and P.~Kabamba, ``Energy-optimal path planning for solar-powered
  aircraft in level flight,'' in \emph{AIAA Paper 2007-6655}.

\bibitem{klesh2009solar}
A.~T. Klesh and P.~T. Kabamba, ``Solar-powered aircraft: Energy-optimal path
  planning and perpetual endurance,'' \emph{J Guid Control Dyn}, vol.~32,
  no.~4, pp. 1320--1329, 2009.

\bibitem{hosseini2016energy}
S.~Hosseini and M.~Mesbahi, ``Energy-aware aerial surveillance for a
  long-endurance solar-powered unmanned aerial vehicles,'' \emph{J Guid Control
  Dyn}, vol.~39, no.~9, pp. 1980--1993, 2016.

\bibitem{dudek2013hybrid}
M.~Dudek and et~al., ``{Hybrid fuel cell--battery system as a main power unit
  for small unmanned aerial vehicles (UAV)},'' \emph{Int. J. Electrochem.
  Sci.}, vol.~8, no.~6, pp. 8442--8463, 2013.

\bibitem{mobariz2015long}
K.~N. Mobariz, A.~M. Youssef, and M.~Abdel-Rahman, ``Long endurance hybrid fuel
  cell-battery powered uav,'' \emph{World Journal of Modelling and Simulation},
  vol.~11, no.~1, pp. 69--80, 2015.

\bibitem{dobrokhodov2020energy}
V.~Dobrokhodov, K.~D. Jones, C.~Walton, and I.~I. Kaminer, ``Energy-optimal
  trajectory planning of hybrid ultra-long endurance uav in time-varying energy
  fields,'' in \emph{AIAA Paper 2020-2299}.

\bibitem{sundar2016formulations}
K.~Sundar, S.~Venkatachalam, and S.~Rathinam, ``Formulations and algorithms for
  the multiple depot, fuel-constrained, multiple vehicle routing problem,'' in
  \emph{IEEE-ACC}, 2016, pp. 6489--6494.

\bibitem{sundar2017analysis}
------, ``Analysis of mixed-integer linear programming formulations for a
  fuel-constrained multiple vehicle routing problem,'' \emph{Unmanned Systems},
  vol.~5, no.~04, pp. 197--207, 2017.

\bibitem{erdelic2019survey}
T.~Erdeli{\'c} and T.~Cari{\'c}, ``A survey on the electric vehicle routing
  problem: variants and solution approaches,'' \emph{J. Adv. Transp.}, vol.
  2019, 2019.

\bibitem{verma2018electric}
A.~Verma, ``Electric vehicle routing problem with time windows, recharging
  stations and battery swapping stations,'' \emph{EJTL}, vol.~7, no.~4, pp.
  415--451, 2018.

\bibitem{doppstadt2016hybrid}
C.~Doppstadt, A.~Koberstein, and D.~Vigo, ``The hybrid electric
  vehicle--traveling salesman problem,'' \emph{Eur. J. Oper. Res.}, vol. 253,
  no.~3, pp. 825--842, 2016.

\bibitem{vincent2017simulated}
F.~Y. Vincent, A.~P. Redi, Y.~A. Hidayat, and O.~J. Wibowo, ``A simulated
  annealing heuristic for the hybrid vehicle routing problem,'' \emph{Appl.
  Soft Comput.}, vol.~53, pp. 119--132, 2017.

\bibitem{hiermann2019routing}
G.~Hiermann, R.~F. Hartl, J.~Puchinger, and T.~Vidal, ``Routing a mix of
  conventional, plug-in hybrid, and electric vehicles,'' \emph{Eur. J. Oper.
  Res.}, vol. 272, no.~1, pp. 235--248, 2019.

\bibitem{sundar2013algorithms}
K.~Sundar and S.~Rathinam, ``Algorithms for routing an unmanned aerial vehicle
  in the presence of refueling depots,'' \emph{T-ASE}, vol.~11, no.~1, pp.
  287--294, 2013.

\bibitem{alyassi2022autonomous}
R.~Alyassi, M.~Khonji, A.~Karapetyan, S.~C.-K. Chau, K.~Elbassioni, and C.-M.
  Tseng, ``Autonomous recharging and flight mission planning for
  battery-operated autonomous drones,'' \emph{T-ASE}, vol.~20, no.~2, pp.
  1034--1046, 2022.

\bibitem{jin2006optimal}
Z.~Jin, T.~Shima, and C.~J. Schumacher, ``Optimal scheduling for refueling
  multiple autonomous aerial vehicles,'' \emph{T-RO}, vol.~22, no.~4, pp.
  682--693, 2006.

\bibitem{mathew2015multirobot}
N.~Mathew, S.~L. Smith, and S.~L. Waslander, ``Multirobot rendezvous planning
  for recharging in persistent tasks,'' \emph{T-RO}, vol.~31, no.~1, pp.
  128--142, 2015.

\bibitem{pugliese2013survey}
L.~D.~P. Pugliese and F.~Guerriero, ``A survey of resource constrained shortest
  path problems: Exact solution approaches,'' \emph{Networks}, vol.~62, no.~3,
  pp. 183--200, 2013.

\bibitem{cabrera2020exact}
N.~Cabrera, A.~L. Medaglia, L.~Lozano, and D.~Duque, ``An exact bidirectional
  pulse algorithm for the constrained shortest path,'' \emph{Networks},
  vol.~76, no.~2, pp. 128--146, 2020.

\bibitem{lozano2016exact}
L.~Lozano, D.~Duque, and A.~L. Medaglia, ``An exact algorithm for the
  elementary shortest path problem with resource constraints,''
  \emph{Transportation Science}, vol.~50, no.~1, pp. 348--357, 2016.

\bibitem{lozano2013exact}
L.~Lozano and A.~L. Medaglia, ``On an exact method for the constrained shortest
  path problem,'' \emph{Comput Oper Res}, vol.~40, no.~1, pp. 378--384, 2013.

\bibitem{bolivar2014acceleration}
M.~A. Bol{\'\i}var, L.~Lozano, and A.~L. Medaglia, ``Acceleration strategies
  for the weight constrained shortest path problem with replenishment,''
  \emph{Optimization Letters}, vol.~8, no.~8, pp. 2155--2172, 2014.

\bibitem{smith2012solving}
O.~J. Smith, N.~Boland, and H.~Waterer, ``Solving shortest path problems with a
  weight constraint and replenishment arcs,'' \emph{Computers \& Operations
  Research}, vol.~39, no.~5, pp. 964--984, 2012.

\bibitem{cabral2007network}
E.~A. Cabral, E.~Erkut, G.~Laporte, and R.~A. Patterson, ``The network design
  problem with relays,'' \emph{Eur. J. Oper. Res.}, vol. 180, no.~2, pp.
  834--844, 2007.

\bibitem{cabral2008wide}
------, ``Wide area telecommunication network design: application to the
  alberta supernet,'' \emph{JORS}, vol.~59, no.~11, pp. 1460--1470, 2008.

\bibitem{cabralPhD}
E.~A. Cabral, ``Wide area telecommunication network design: problems and
  solution algorithms with application to the alberta supernet,'' Ph.D.
  dissertation, University of Alberta, 2005.

\bibitem{manyam2022path}
S.~G. Manyam, D.~W. Casbeer, S.~Darbha, I.~E. Weintraub, and K.~Kalyanam,
  ``Path planning and energy management of hybrid air vehicles for urban air
  mobility,'' \emph{RA-L}, vol.~7, no.~4, pp. 10\,176--10\,183, 2022.

\bibitem{jadischke2023optimal}
J.~H. Jadischke, M.~Wolff, J.~Zumberge, B.~Hencey, and A.~Ngo, ``Optimal route
  planning and power management for hybrid {UAV} using {A*} algorithm,'' in
  \emph{AIAA Paper 2023-4508}.

\bibitem{scott2022hybrid}
D.~Scott, S.~G. Manyam, D.~W. Casbeer, M.~Kumar, M.~Rothenberger, and I.~E.
  Weintraub, ``Power management for noise aware path planning of hybrid
  {UAVs},'' in \emph{IEEE-ACC}, 2022.

\bibitem{mitchell2002branch}
J.~E. Mitchell, ``Branch-and-cut algorithms for combinatorial optimization
  problems,'' \emph{Handbook of Applied Optimization}, vol.~1, pp. 65--77,
  2002.

\bibitem{handler1980dual}
G.~Y. Handler and I.~Zang, ``A dual algorithm for the constrained shortest path
  problem,'' \emph{Networks}, vol.~10, no.~4, pp. 293--309, 1980.

\bibitem{desrosiers1995time}
J.~Desrosiers, Y.~Dumas, M.~M. Solomon, and F.~Soumis, ``Time constrained
  routing and scheduling,'' \emph{Handbooks in Operations Research and
  Management Science}, vol.~8, pp. 35--139, 1995.

\bibitem{brumbaugh1989empirical}
J.~Brumbaugh-Smith and D.~Shier, ``An empirical investigation of some
  bicriterion shortest path algorithms,'' \emph{Eur. J. Oper. Res.}, vol.~43,
  no.~2, pp. 216--224, 1989.

\bibitem{cplex2009v12}
IBM-ILOG, ``{CPLEX} v12.9,'' \emph{International Business Machines
  Corporation}, 2009.

\end{thebibliography}


\end{document}